\documentclass{amsart}
\usepackage[noadjust]{cite}
\usepackage{amsfonts,amsthm,amsmath,mathrsfs}
\usepackage{mathtools,verbatim}
\usepackage{enumerate}

\theoremstyle{plain}
\newtheorem{theorem}{Theorem}
\newtheorem*{mainthm}{Main Theorem}
\newtheorem{lemma}{Lemma}
\newtheorem{claim}{Claim}

\newtheorem{corollary}{Corollary}
\newtheorem{question}{Question}

\theoremstyle{definition}

\newcommand{\de}{\delta}

\newcommand{\al}{\alpha}
\newcommand{\ga}{\gamma}

\newcommand{\R}{\mathbb{R}}
\newcommand{\Q}{\mathbb{Q}}
\newcommand{\Z}{\mathbb{Z}}
\newcommand{\N}{\mathbb{N}}

\newcommand{\intpart}[1]{\left\lfloor #1 \right\rfloor}
\newcommand{\fracpart}[1]{\left\{ #1 \right\}}
\newcommand{\fracdist}[1]{\left\| #1 \right\|}
\newcommand{\bothceil}[1]{\left\lceil #1 \right\rceil}
\newcommand{\bothfloor}[1]{\left\lfloor #1 \right\rfloor}
\newcommand{\psa}{\text{PS}(\al)}
\newcommand{\ps}[1]{\text{PS}(#1)}
\newcommand{\varchange}[2]{t_{#1} \left({#2} \right)}
\newcommand{\varchangeempty}[1]{t_{#1}}
\newcommand{\varchangeprime}[2]{t'_{#1}({#2})}
\newcommand{\intone}{\theta_1}
\newcommand{\inttwo}{\theta_2}
\newcommand{\intthree}{\theta_3}
\newcommand{\fs}[1]{\text{FS}\big({#1}\big)}
\newcommand{\ip}{\text{IP}}
\newcommand{\vip}{\text{VIP}}
\newcommand{\ipzero}{\text{IP}_0}

\newcommand\restr[2]{{ \left.\kern-\nulldelimiterspace #1 \right|_{#2}}}

\author[D. Glasscock]{Daniel Glasscock}
\address{Department of Mathematics, The Ohio State University, 231 W. 18th Ave., Columbus, OH 43210}
\email{glasscock.4@math.osu.edu}

\title[Solutions to linear equations in Piatetski-Shapiro sequences]{Solutions to certain linear equations in Piatetski-Shapiro sequences}

\begin{document}

\begin{abstract}
Denote by $\psa$ the image of the Piatetski-Shapiro sequence $n \mapsto \intpart {n^\al}$, where $\al > 1$ is non-integral and $\intpart x$ is the integer part of $x \in \R$. We partially answer the question of which bivariate linear equations have infinitely many solutions in $\psa$: if $a, b \in \R$ are such that the equation $y=ax+b$ has infinitely many solutions in the positive integers, then for Lebesgue-a.e. $\al > 1$, it has infinitely many or at most finitely many solutions in $\psa$ according as $\al < 2$ (and $0 \leq b < a$) or $\al > 2$ (and $(a,b) \neq (1,0)$). We collect a number of interesting open questions related to further results along these lines.
\end{abstract}



\maketitle

\section{Introduction}\label{sec:intro}

A \emph{Piatetski-Shapiro sequence} is a sequence of the form $(\intpart {n^\al})_{n \in \N}$ for non-integral $\al > 1$, where $\intpart x$ is the integer part of $x \in \R$  and $\N$ is the set of positive integers. Denote by $\psa$ the image of $n \mapsto \intpart {n^\al}$. We will say that the linear equation
\begin{align}\label{eqn:main}y=ax+b, \quad a, b \in \R\end{align}
is \emph{solvable} in $\psa$ if there are infinitely many distinct pairs $(x,y) \in \psa \times \psa$ satisfying (\ref{eqn:main}), and \emph{unsolvable} otherwise. This terminology extends as expected to solving equations and systems of equations in other subsets of $\N$.

\begin{theorem}\label{thm:maintheorem}
Suppose that (\ref{eqn:main}) is solvable in $\N$. For Lebesgue-a.e. $\al > 1$,
\begin{enumerate}[i.]
\item if $\al < 2$ and $0 \leq b < a$, then (\ref{eqn:main}) is solvable in $\psa$;
\item if $\al > 2$ and $(a,b) \neq (1,0)$, then (\ref{eqn:main}) is unsolvable in $\psa$.
\end{enumerate}
\end{theorem}

Piatetski-Shapiro sequences get their name from Ilya Piatetski-Shapiro, who proved a Prime Number Theorem for $(\intpart {n^\al})_{n \in \N}$ for all $1 < \al < 12/11$; see \cite{psoriginal}. Similar results regarding the distribution of $(\intpart {n^\al})_{n \in \N}$ in arithmetic progressions and the square-free numbers hold for various ranges of $\al$ in both metrical and complete versions; see \cite{baker} for recent results in this direction and further references.

The motivation for this work comes from another line of thought. Since $\psa$ is the (rounded) image of $\N$ under the \emph{Hardy field function}\footnote{Any subfield of the ring of germs at $+\infty$ of continuous real-valued functions on $\R$ which is closed under differentiation is a \emph{Hardy field}, and its members are \emph{Hardy field functions}.} $x \mapsto x^\al$, it is known to be a so-called \emph{set of multiple recurrence} in ergodic theory (see \cite{frantzikinakiswierdl}); thus, for example, every $E \subseteq \N$ with $\limsup_{N \to \infty} \big| E \cap \{1,\ldots,N\} \big| \big/ N > 0$ contains arbitrarily long arithmetic progressions with step size in $\psa$. That $\psa$ is a set of multiple recurrence follows from it containing ``many divisible polynomial patterns'' (see \cite{frantzikinakiswierdl}, Section 5); in particular, when $1 < \al < 2$, the set $\psa$ contains arbitrarily long arithmetic progressions and arithmetic progressions of every sufficiently large step.

Another class of sets which enjoy strong recurrence properties are those which possess \emph{$\ip$-structure}. A \emph{finite sums} set in $\N$ is a set of the form
\begin{align}\label{eqn:fsdef}\fs {(x_i)_{i=1}^n} = \left\{ \sum_{i \in I} x_i \ \middle| \ I \subseteq \{1, \ldots, n\}, \ I \neq \emptyset \right\},\end{align}
where $(x_i)_{i=1}^n \subseteq \N$, and a set $A \subseteq \N$ is called $\ipzero$ if it contains arbitrarily large finite sums sets. Finite sums sets define ``linear'' $\ip$-structure; we define a higher order analogue, $\vip$-structure, in Section \ref{sec:remarks}. Sets with $\vip$-structure are known to be sets of multiple recurrence in both topological and measure-theoretic dynamical systems; see, for example, \cite{FK,BM,BMultrafilter,BFM}.

It would be interesting, therefore, to identify $\ip$-structure in sequences arising from Hardy field functions, and the Piatetski-Shapiro sequences provide an ideal first candidate in this search: they are easily described and already known to form sets of multiple recurrence. We elaborate on this further in Section \ref{sec:remarks}.

In attempting to find $\ip$-structure, a more basic question arises that does not seem to have been addressed in the literature: which linear equations are solvable in $\psa$?  In some cases, this question can already be easily answered. For example, for all $1 < \al < 2$, because $\psa$ contains arbitrarily long arithmetic progressions, it contains infinitely many solutions to balanced, homogeneous linear equations: $\mathbf{a} \cdot \mathbf{x} = 0$ where $\sum_i a_i = 0$. Because $\psa$ contains progressions of every sufficiently large step, it contains solutions to linear equations such as $x+y=z$, as well.

Whether other simple linear equations, such as $y=2x$, are solvable in $\psa$ does not follow from the aforementioned results. Theorem \ref{thm:maintheorem} serves as a partial answer to the question of which bivariate linear equations are solvable in $\psa$. As a corollary, we find sets of the form $\fs {(x_i)_{i=1}^3}$ in $\psa$ for Lebesgue-a.e. $1 < \al < 2$; this is a famous open problem in the set of squares, $\ps 2$.

\subsection*{Notation} For $x \in \R$, denote the distance to the nearest integer by $\fracdist x$, the fractional part by $\fracpart x$, the integer part (or floor) by $\bothfloor x$, and the ceiling by $\bothceil x := -\bothfloor {-x}$. Denote the Lebesgue measure on $\R$ by $\lambda$, and denote the set of those points belonging to infinitely many of the sets in the sequence $(E_n)_{n \in \N}$ by $\limsup_{n \to\infty} E_n$. Given two positive-valued functions $f$ and $g$, we write $f \ll_{a_1,\ldots,a_k} g$ or $g \gg_{a_1,\ldots,a_k} f$ if there exists a constant $K > 0$ depending only on the quantities $a_1, \ldots, a_k$ for which $f(x) \leq K g(x)$ for all $x$ in the domain common to both $f$ and $g$.

\subsection*{Acknowledgements} The author gives thanks to Professor Vitaly Bergelson for the inspiration for this work and to the referee for several helpful suggestions.

\section{Reduction to Diophantine approximation}\label{sec:reduction} 

We prove Theorem \ref{thm:maintheorem} by reducing it to the following theorem in Diophantine approximation.

\begin{theorem}\label{thm:firstDAform}
Let $I \subseteq [0,1)$ be a set with non-empty interior, $a, c > 0$, $a \neq 1$, and $\ga \in \R \setminus \{0\}$. For Lebesgue-a.e. $\al > 1$, the system
\begin{align}\label{eqn:systemone}\begin{dcases} \fracdist {a^{1/\al}n} \leq \frac{c}{n^{\al-1}} \\ \fracpart {\ga n^\al} \in I \end{dcases}\end{align}
is solvable or unsolvable in $\N$ according as $\al < 2$ or $\al > 2$.
\end{theorem}

This theorem can be seen as ``twisted'' Diophantine approximation. Indeed, when $\al - 1 < 1$, the first inequality in (\ref{eqn:systemone}) is solvable in $\N$ by Dirichlet's Theorem; more concretely, any sufficiently large denominator of a continued fraction convergent of $a^{1/\al}$ will yield a solution. The second condition in (\ref{eqn:systemone}) provides the twist.

\begin{proof}[Proof of Theorem \ref{thm:maintheorem} assuming Theorem \ref{thm:firstDAform}]
Note that (\ref{eqn:main}) is solvable in $\N$ if and only if
\begin{align*}a, b \in \Q, \ a = \frac{a_1}{a_2}, \ a_1, a_2 \in \N, \ (a_1, a_2) = 1, \ \text{and } a_2b \in \Z.\end{align*}
Combined with the fact that consecutive differences in $\psa$ tend to infinity, the theorem holds when $a = 1$.  Assume $a \neq 1$. Note that there exists an integer $0 \leq d \leq a_2-1$ for which 
\[a \intpart {n^\al} + b\in \Z \iff \intpart{n^\al} \equiv d \pmod {a_2} \iff \fracpart {\frac {n^\al}{a_2}} \in \left[ \frac{d}{a_2}, \frac{d+1}{a_2} \right).\]
It follows that
\begin{align}
\notag a \intpart {n^\al} + b\in \psa \iff & \exists \ k \in \N, \ a \intpart {n^\al} + b = \intpart{k^\al}\\
\notag \iff & \begin{cases} a \intpart {n^\al} + b \in \Z \ \text{ and } \\ \exists \ k \in \N, \ a \intpart {n^\al} + b \leq k^\al < a \intpart {n^\al} + b + 1 \end{cases} \\
\label{eqn:suffandnecconditions}\iff & \fracpart {\frac {n^\al}{a_2}} \in \left[ \frac{d}{a_2}, \frac{d+1}{a_2} \right) \ \text{ and } J_n \cap \N \neq \emptyset,
\end{align}
where, by the Mean Value Theorem,
\begin{align*}J_n &= \left[ \left(a \intpart {n^\al} + b\right)^{1/ \al},\left(a \intpart {n^\al} + b + 1\right)^{1/ \al} \right) = a^{1/ \al}n + [L_n,R_n),\\
L_n &= -\frac a\al \left(\fracpart {n^\al} - \frac ba \right) l_n^{-1 + 1/\al}, \ \ l_n \text{ between } a n^\al \text{ and } a \intpart {n^\al} + b,\\
R_n &= \frac a\al \left( \frac {b+1}a - \fracpart {n^\al} \right) r_n^{-1 + 1/\al}, \ \ r_n \text{ between } a n^\al \text{ and } a \intpart {n^\al} + b+1.\end{align*}
Note that $J_n, L_n, R_n, l_n,$ and $r_n$ all depend on $\al$. This shows so far that (\ref{eqn:main}) is solvable in $\psa$ if and only if the system in (\ref{eqn:suffandnecconditions}) is solvable in $\N$.

We proceed by showing that solutions to (\ref{eqn:systemone}) yield solutions to (\ref{eqn:suffandnecconditions}) and vice versa when $I$, $c$, and $\ga$ are chosen appropriately. To this end, for $i=1,2$, let
\begin{align*}A &= \{\al > 1 \ | \ (\ref{eqn:suffandnecconditions}) \text{ is solvable in } \N \},\\
B_i &= \{\al > 1 \ | \ (\ref{eqn:systemone}) \text{ is solvable in } \N \text{ for } I_i, c_i, \ga_i, a \}.\end{align*}
To prove Theorem \ref{thm:maintheorem}, it suffices by Theorem \ref{thm:firstDAform} to find $I_i, c_i, \ga_i$, $i=1,2$, for which
\begin{align}\label{eqn:containmenttoprove}B_1 \cap (1,2) \subseteq A \subseteq B_2.\end{align}

We begin with the first containment in (\ref{eqn:containmenttoprove}), which we will show under the assumption that $0 \leq b < a$. Let $I_0$ be the middle third sub-interval of the interval $\big (b/a, \min(1,(b+1)/a) \big)$. Let $I_1 = d/a_2 + I_0/a_2$, $\ga_1 = 1/a_2$, and $c_1$ be a constant depending only on $a$ and $b$ to be specified momentarily.

Suppose $\al \in B_1 \cap (1,2)$ and that $n$ is a solution to (\ref{eqn:systemone}); we will show that if $n$ is sufficiently large, then it solves the system in (\ref{eqn:suffandnecconditions}). By (\ref{eqn:systemone}),
\[\fracpart {\frac {n^\al}{a_2}} = \fracpart {\ga_1 n^\al} \in I_1 \subseteq \left[ \frac d{a_2}, \frac {d+1}{a_2} \right).\]
This also implies that $\fracpart {n^\al} \in I_0$, so
\[\fracpart {n^\al} - \frac ba \gg_{a,b} 1, \quad \frac {b+1}a - \fracpart {n^\al} \gg_{a,b} 1.\]
Combining these estimates with the facts that $\al \in (1,2)$ and, for $n$ sufficiently large, $a \intpart {n^\al} + b + 1 \leq 2 a n^\al$, we get
\begin{align*}
-L_n &= \frac a\al \left(\fracpart {n^\al} - \frac ba \right) l_n^{-1 + 1/\al} \gg_{a,b} \frac 1{n^{\al-1}},\\
R_n &= \frac a\al \left( \frac {b+1}a - \fracpart {n^\al} \right) r_n^{-1 + 1/\al} \gg_{a,b} \frac 1{n^{\al-1}}.
\end{align*}
Set $c_1$ to be half the minimum of the constants implicit in these two expressions. It follows that $J_n$ contains an open interval centered at $a^{1/\al}n$ of length $2 c_1 \big / n^{\al - 1}$. By (\ref{eqn:systemone}), $\fracdist {a^{1/\al}n} \leq c_1 \big/ n^{\al-1}$, so the interval $J_n$ contains the nearest integer to $a^{1/\al}n$; in particular, $J_n \cap \N \neq \emptyset$, so $n$ solves (\ref{eqn:suffandnecconditions}).

The second containment in (\ref{eqn:containmenttoprove}) is handled similarly. Let $I_2 = [0,1)$, $\ga_2 = 1$, and $c_2$ be a constant depending only on $a$ to be specified momentarily. Suppose that $\al \in A$ and $n$ solves (\ref{eqn:suffandnecconditions}); we will show that $n$ satisfies (\ref{eqn:systemone}).  The second condition in (\ref{eqn:systemone}) is satisfied automatically by our choice of $I_2$. For $n$ sufficiently large, $a \intpart {n^\al} + b \geq a n^\al \big / 2$, whereby $|L_n|, |R_n| \leq c_2 \big/ 2 {n^{\al-1}}$, where $c_2$ is chosen (depending only on $a$) to satisfy both inequalities. Since $J_n$ contains an integer, it must be that $\left\|a^{1/\al}n\right\| \leq c_2 \big/ n^{\al-1}$, meaning $n$ satisfies (\ref{eqn:systemone}).
\end{proof}

\section{Proof of Theorem \ref{thm:firstDAform}} 

To prove Theorem \ref{thm:firstDAform}, we first change variables under $\varchange ax = (\log_ax)^{-1}$ to arrive at the equivalent Theorem \ref{thm:secondDAform}.  Proof of the equivalence of these two theorems is a routine exercise using the fact that $\varchangeempty a$ is measure-theoretically non-singular. The lemmas used in the proof of the following theorem may be found in Section \ref{sec:supportinglemmata}.

\begin{theorem}\label{thm:secondDAform} 
Let $I \subseteq [0,1)$ be a set with non-empty interior, $a, c > 0$, $a \neq 1$, and $\ga \in \R \setminus \{0\}$. If $a < 1$, then for Lebesgue-a.e. $a < \theta < 1$, the system 
\begin{align}\label{eqn:systemtwo}\begin{dcases} \fracdist {\theta n} \leq \frac{c}{n^{\varchange a\theta -1}} \\ \fracpart {\ga n^{\varchange a\theta }} \in I \end{dcases}\end{align}
is solvable or unsolvable in $\N$ according as $\theta < \sqrt a$ or $\theta > \sqrt a$. If $a > 1$, then for Lebesgue-a.e. $1 < \theta < a$, the system is solvable or unsolvable in $\N$ according as $\theta > \sqrt a$ or $\theta < \sqrt a$.
\end{theorem}

\begin{proof} Fix $I$, $a$, $c$, and $\ga$. Suppose $a < 1$; the case $a > 1$ follows from the proof below with the obvious modifications. Without loss of generality, we may assume that $I$ is an interval with non-empty interior. For brevity, we will suppress dependence on $I$, $a$, $c$, and $\ga$ in the asymptotic notation appearing in the proof. 

Let $\Theta \subseteq (a,1)$ be the set of those $\theta$ satisfying the conclusion of the theorem. We will show that $\Theta$ is of full Lebesgue measure by showing that it has full measure in the intervals $(a,\sqrt a)$ and $(\sqrt a,1)$.

To show that $\Theta \cap (a,\sqrt a)$ is of full measure, it suffices by Lemma \ref{lem:densitylemma} to show that there exists a $\de > 0$ such that for all $a < \intone < \inttwo < \sqrt a$,
\begin{align}\label{eqn:biggerthandelta}\lambda \left(\Theta \cap (\intone,\inttwo)\right) \geq \de (\inttwo - \intone).\end{align}
To this end, fix $a < \intone < \inttwo < \sqrt a$. In what follows, the phrase ``for all sufficiently large $n$'' means ``for all $n \geq n_0$,'' where $n_0 \in \N$ may depend on any of the parameters introduced so far, including $\intone$ and $\inttwo$.

For $n \in \N$, define
\begin{align*}
E_n &= \left\{ \theta \in (\intone,\inttwo) \ \middle | \ \fracdist {\theta n} \leq \psi(n) \right\}, & \psi (n) &= \frac{1}{n},\\
F_n &= \left\{ \theta \in (\intone,\inttwo) \ \middle | \ \fracpart {\ga n^{\varchange a\theta }} \in I \right\}, & G_n &= E_n \cap F_n.
\end{align*}
For $\theta \in (\intone,\inttwo)$, $\varchange a\theta < \varchange a\inttwo < 2$, so for sufficiently large $n$, we have $\psi(n)\leq c \big/ n^{\varchange a\theta -1}$. It follows that $\limsup_{n \to\infty}G_n \subseteq \Theta \cap (\intone, \inttwo)$. Therefore, in order to show (\ref{eqn:biggerthandelta}), it suffices to prove that there exists a $\de > 0$, independent of $\intone, \inttwo$, for which
\begin{align}\label{eqn:mainboundonlimsup}\lambda \left( \limsup_{\substack{p \to \infty\\ p \text{ prime}}} G_p \right) \geq \de (\inttwo - \intone).\end{align}
Passing to primes here makes parts of the argument technically easier. To ease notation, any sum indexed over $p$ or $q$ will be understood to be a sum over prime numbers.

To prove (\ref{eqn:mainboundonlimsup}), it suffices by Lemma \ref{lem:advancedBC} to prove that
\begin{align}\label{eqn:Gpsarebigenough} \sum_{p = 2}^\infty \lambda(G_p) = \infty\end{align}
and that there exists a $\de > 0$, independent of $\intone, \inttwo$ for which
\begin{align}\label{eqn:Gpsareindependent} \limsup_{N \to \infty} \left( \sum_{p = 2}^N \lambda(G_p) \right)^2 \left( \sum_{p,q = 2}^N \lambda(G_p \cap G_q) \right)^{-1} \geq \de (\inttwo - \intone).\end{align}

First we show (\ref{eqn:Gpsarebigenough}) using Lemma \ref{lem:equid}. Fix $0 < \eta < \min \big( \intone, 1-\inttwo, (\inttwo-\intone)/3 \big)$. For $n \in \N$, let
\begin{align}\label{eqn:defoftn}
\begin{split} S_n &= \big\{ m \in \Z \ \big| \ \intone + \eta < m / n < \inttwo-\eta \big\},\\
T_n &= \big\{ m \in \Z \ \big| \ \intone - \eta < m / n < \inttwo+\eta \big\}, \end{split}
\end{align}
and note that for $n$ sufficiently large,
\begin{align}\label{eqn:boundsonsets} (\inttwo-\intone)n \ll |S_n| < |T_n| \ll (\inttwo-\intone)n. \end{align}
For $n$ sufficiently large, the set $E_n$ may be approximated by a disjoint union of intervals:
\[\bigcup_{m \in S_n} E_{n,m} \subseteq E_n \subseteq \bigcup_{m \in T_n} E_{n,m},\]
where
\[E_{n,m} = \left[ \frac mn - \frac{\psi(n)}{n}, \frac mn + \frac{\psi(n)}{n} \right]. \]
Let $I_0 \subseteq I$ be the middle third sub-interval of $I$.

\begin{claim}\label{lem:inInaughtimpliesinI}
For $n$ sufficiently large and $m \in S_n$, if $\fracpart {\ga n^{\varchange a{m/n} }} \in I_0$, then $E_{n,m} \subseteq F_n$.
\end{claim}

\begin{proof}
Note first that $\left|\varchangeprime ax \right| \ll 1$ for $x \in (a, \sqrt a)$. For $\theta \in E_{n,m}$, by applying the Mean Value Theorem twice, we see that for some $\xi_1$ between $\varchange a\theta$ and $\varchange a{m/n}$ and some $\xi_2$ between $\theta$ and $m/n$,
\begin{align*}
\left|\ga n^{\varchange a\theta } - \ga n^{\varchange a{m/n} } \right| &= \left| \ga \right| \left| \varchange a\theta - \varchange a{m/n} \right| \log n \ n^{\xi_1} \\
&\leq |\ga| \left| \theta - \frac mn \right| \big| \varchangeprime a{\xi_2}\big| \log n \ n^{\varchange a\inttwo}\\
&\ll \frac{\psi(n)}n \log n \ n^{\varchange a\inttwo}\\
&\leq \frac{\log n}{n^{2 - \varchange a\inttwo}}.
\end{align*}
Since $2 - \varchange a\inttwo > 0$, for $n$ sufficiently large,
\[\left|\ga n^{\varchange a\theta } - \ga n^{\varchange a{m/n} } \right| < \frac{\lambda(I)}{3}.\]
Therefore, if $\fracpart {\ga n^{\varchange a{m/n} }} \in I_0$, then for all $\theta \in E_{n,m}$, $\fracpart {\ga n^{\varchange a\theta }} \in I$, i.e. $E_{n,m} \subseteq F_n$.
\end{proof}

By the equidistribution result in Lemma \ref{lem:equid}, for $n$ sufficiently large,
\[\frac{\left| \left\{ m \in S_n \ \middle| \ \fracpart {\ga n^{\varchange a{m/n} }} \in I_0 \right\} \right|}{(\inttwo-\intone-2\eta)n} \geq \frac{\lambda(I_0)}{2}.\]
Combining this with Claim \ref{lem:inInaughtimpliesinI} and the bounds on $\eta$, there are $\gg (\inttwo-\intone)n$ integers $m \in S_n$ for which $E_{n,m} \subseteq F_n$. It follows that for $n$ sufficiently large,
\begin{align}\label{eqn:gnisbig}\lambda(G_n) \gg (\inttwo - \intone)n \frac{\psi(n)}{n} = (\inttwo - \intone) \frac 1{n}.\end{align}
This proves (\ref{eqn:Gpsarebigenough}).

Now we show (\ref{eqn:Gpsareindependent}) by estimating the ``overlaps'' between the $G_p$'s. It suffices to prove that there exists a constant $K$ (which may depend on any of the parameters introduced so far) such that for all sufficiently large primes $p$ and for all $N \geq p$,
\begin{align}\label{eqn:sufficientupperboundcondition}
\sum_{q > p}^N \lambda (E_p \cap E_q) \ll (\inttwo-\intone) \sum_{q > p}^N \psi(p) \psi(q) + K \psi(p).
\end{align}
Indeed, suppose (\ref{eqn:gnisbig}) and (\ref{eqn:sufficientupperboundcondition}) both hold for all primes $p$ greater than some sufficiently large $p_0 \in \N$. Using the trivial bound $\lambda(G_p \cap G_q) \leq \lambda(E_p \cap E_q)$, it follows that
\begin{align*}
\sum_{p, q = 2}^N \lambda(G_p \cap G_q) &\leq 2\left( \sum_{\substack{p \geq p_0 \\ q > p}}^N \lambda(G_p \cap G_q) + \sum_{\substack{p < p_0 \\ q > p}}^N \lambda(G_q) \right) + \sum_{q = 2}^N \lambda(G_q) \\
&\ll (\inttwo-\intone)\sum_{\substack{p \geq p_0 \\ q > p}}^N \psi(p) \psi(q)  + K \sum_{p\geq p_0}^N \psi(p) + \sum_{\substack{p < p_0 \\ q \geq p}}^N \lambda(G_q)\\
&\ll \frac 1{\inttwo-\intone} \sum_{p, q = 2}^N \lambda(G_p) \lambda(G_q)  + p_0 \left(\frac{K}{\inttwo-\intone}+1\right) \sum_{q = 2}^N \lambda(G_q),
\end{align*}
where the last line follows from (\ref{eqn:gnisbig}). This combines with (\ref{eqn:Gpsarebigenough}) to yield (\ref{eqn:Gpsareindependent}).

To show (\ref{eqn:sufficientupperboundcondition}), note that the set $E_p$ is covered by a union of intervals $E_{p,r}$, each of length $2 \psi(p) \big / p$. If $p < q$ and $E_{p,r} \cap E_{q,s} \neq \emptyset$, then by estimating the distance between the midpoints of the intervals,
\begin{align*}
\left| sp - rq \right| &< pq \ 3 \max \left( \frac{\psi(p)}p, \frac{\psi(q)}q \right) = 3 q \psi(p),\\
\lambda \left( E_{p,r} \cap E_{q,s} \right) &\leq \min \left( 2 \frac {\psi(p)}{p}, 2 \frac {\psi(q)}{q} \right) = 2 \frac {\psi(q)}{q}.
\end{align*}
The left hand side of (\ref{eqn:sufficientupperboundcondition}) is then
\begin{align*}
\sum_{q > p} \lambda(E_p \cap E_q) = \sum_{q > p} \sum_{\substack{r \in T_p \\ s \in T_q}} \lambda \left( E_{p,r} \cap E_{q,s} \right) \ll \sum_{q > p} \frac {\psi(q)}{q} \sum_{\substack{r \in T_p, \ s \in T_q \\ |sp-rq| < 3 q \psi(p)}} 1.
\end{align*}
Now (\ref{eqn:sufficientupperboundcondition}) will follow by partitioning the range of the sum on $q$ and applying Lemma \ref{lem:solutionstoeuclideanproblem}. Indeed, the right hand side of the previous expression is equal to
\begin{align*}
\sum_{\ell = 0}^\infty \sum_{\substack{2^\ell p < q < 2^{\ell+1} p \\ q \leq N}} \frac {\psi(q)}{q} \sum_{\substack{r \in T_p, \ s \in T_q \\ |sp-rq| < 3 q \psi(p)}} 1 &\leq \sum_{\ell = 0}^\infty \frac {\psi(2^\ell p)}{2^\ell p} \sum_{\substack{Q_\ell < q < \min(2Q_\ell, N+1) \\ r \in T_p, \ s \in T_q \\ |sp-rq| < L_\ell}} 1,
\end{align*}
where $L_\ell = 3\cdot 2^{\ell+1} p \psi(p)$ and $Q_\ell = 2^\ell p$. For each $\ell$, we apply Lemma \ref{lem:solutionstoeuclideanproblem} with $N$ and $p$ as they are, $Q_\ell$ as $Q$, $L_\ell$ as $L$, and $(\theta_1-\eta,\theta_2+\eta)$ as $(\eta_1,\eta_2)$: there exists a $K > 0$ depending only on $\theta_1, \theta_2$ (since $\eta$ depends only on $\intone$, $\inttwo$) such that the right hand side of the previous expression is
\begin{align*}
&\ll \sum_{\ell = 0}^\infty \frac {\psi(2^\ell p)}{2^\ell p} \left( (\theta_2-\theta_1 + 2\eta) L_\ell \sum_{\substack{2^\ell p < q < 2^{\ell+1} p \\ q \leq N}} 1 + K 2^\ell p \right)\\
&\ll (\theta_2-\theta_1)  \sum_{\ell = 0}^\infty \sum_{\substack{2^\ell p < q < 2^{\ell+1} p \\ q \leq N}} \psi(p) \psi(2^\ell p) + K \sum_{\ell = 0}^\infty \psi(2^\ell p)\\
&\leq (\theta_2-\theta_1)   \sum_{\ell = 0}^\infty \sum_{\substack{2^\ell p < q < 2^{\ell+1} p \\ q \leq N}} \psi(p) \psi(1/2) \psi(q) + K \psi(p) \sum_{\ell = 0}^\infty \psi(2^\ell)\\
&\ll  (\theta_2-\theta_1) \sum_{q > p}^N \psi(p) \psi(q) + K \psi(p),
\end{align*}
where the third line and fourth lines follow by noting that $\psi(2^\ell p) \leq \psi (q / 2)$, that $\psi$ is multiplicative, and that $\sum_{\ell = 0}^\infty \psi(2^\ell)$ converges. This shows (\ref{eqn:sufficientupperboundcondition}), completing the proof of (\ref{eqn:Gpsareindependent}).

To show that the set $\Theta \cap (\sqrt a,1)$ is of full measure, we will show that for all $\intthree > \sqrt a$, the set $(\intthree,1) \setminus \Theta$ has zero measure. Let
\[H_n = \left\{ \theta \in (\intthree,1) \ \middle | \ \fracdist {\theta n} \leq \frac c{n^{\varchange a\intthree - 1}} \right\}.\]
If $\theta \in (\intthree,1) \setminus \Theta$, then for infinitely many $n\in \N$, $\fracdist {\theta n} \leq c \big/ n^{\varchange a\intthree - 1}$. It follows that
\begin{align}\label{eqn:boundfromabovebylimsup}(\intthree,1) \setminus \Theta \subseteq \limsup_{n \to \infty} H_n.\end{align}
Since $H_n$ is a union of $\ll (1-\intthree)n$ intervals, each of length $2c \big/ n^{\varchange a\intthree}$, and since $\varchange a\intthree > 2$, $\sum_{n=1}^\infty \lambda(H_n) < \infty$.
By the first Borel-Cantelli Lemma, $\limsup_{n \to\infty}H_n$ has zero measure, so $(\intthree,1) \setminus \Theta$ has zero measure by (\ref{eqn:boundfromabovebylimsup}).
\end{proof}

\section{Supporting lemmata}\label{sec:supportinglemmata}  

Here we collect some supporting lemmata. Recall that $\lambda$ denotes the Lebesgue measure on $\R$ and that $\varchange ax = (\log_a x)^{-1}$.

\begin{lemma}[\cite{harmanbook}, Lemma 1.6]\label{lem:densitylemma}
Let $I \subseteq \R$ be an interval and $A \subseteq I$ be measurable.  If there exists a $\de > 0$ such that for every sub-interval $I_0 \subseteq I$, $\lambda (A \cap I_0) \geq \de \lambda (I_0)$, then $A$ is of full measure in $I$: $\lambda( I \setminus A ) = 0$.
\end{lemma}

\begin{lemma}[\cite{harmanbook}, Lemma 2.3]\label{lem:advancedBC}
Let $(X,\mathcal{B},\mu)$ be a measure space with $\mu(X)<\infty$. If $(G_n)_{n \in \N} \subseteq \mathcal{B}$ is a sequence of subsets of $X$ for which $\sum_{n=1}^\infty \mu(G_n) = \infty$, then
\[\mu\left(\limsup_{n\to\infty} G_n \right) \geq \limsup_{N \to \infty} \left( \sum_{n=1}^N \mu(G_n) \right)^2 \left( \sum_{n,m=1}^N \mu(G_n \cap G_m) \right)^{-1}.\]
\end{lemma}

\begin{lemma}\label{lem:solutionstoeuclideanproblem}
Let $N, Q, p \in \N$, $p$ prime with $p \leq Q$, $L > 0$, and $0 < \eta_1 < \eta_2 < 1$. There exists a constant $K > 0$ depending only on $\eta_1$, $\eta_2$ such that the number of triples $(q,r,s) \in \N^3$ satisfying
\[Q < q < \min \big( 2Q, N+1 \big), \ q \text{ prime,} \quad \frac rp, \ \frac sq \in (\eta_1,\eta_2), \quad |sp-rq| < L,\]
is
\[\ll (\eta_2-\eta_1) L \sum_{\substack{ Q < q < \min ( 2Q, N+1 ) \\ q \text{ prime}}} 1 + KQ.\]
\end{lemma}

\begin{proof}
This lemma follows immediately from \cite{harmanbook}, Lemma 6.2, by putting the set of integers $\mathscr{A}$ to be those primes strictly between $Q$ and $\min ( 2Q, N+1 )$ and taking the worst error $KQ$.
\end{proof}

\begin{lemma}\label{lem:equid}
Let $a > 0$, $a \neq 1$, $\gamma \in \R \setminus \{0\}$, and $\min (\sqrt a, a) < \eta_1 < \eta_2 < \max (\sqrt a, a)$. Then
\[\frac{1}{n(\eta_2-\eta_1)} \sum_{\substack{m \in \Z \\ m/n \in (\eta_1,\eta_2)}} \de_{\fracpart {\gamma n^{\varchange a{m/n}}}} \longrightarrow \restr{\lambda}{[0,1)} \quad \text{weakly as } n \to \infty,\]
where $\de_{x}$ denotes the point mass at $x \in [0,1)$.
\end{lemma}

\begin{proof}
Let $N_n = \big| \{ m \in \Z \ | \ m/n \in (\eta_1,\eta_2) \} \big|$, and note that $N_n \big/ \big( n(\eta_2-\eta_1)\big) \to 1$ as $n \to \infty$. For $n, h \in \N$, let
\begin{align*}
g_n(x) = \gamma n^{\varchange a {(x+\bothfloor {\eta_1 n})/ n}}\quad \text{and}\quad g_{n,h}(x) = g_n(x+h) - g_n(x).
\end{align*}
In this notation, we must show
\begin{align}\label{eqn:equidistributiontoshow}\frac{1}{N_n} \sum_{i=1}^{N_n} \de_{\fracpart {g_n(i)}} \longrightarrow \restr{\lambda}{[0,1)} \quad \text{weakly as } n \to \infty.\end{align}

Since $\varchangeempty a$ is either increasing ($a<1$) or decreasing ($a>1$) between $\sqrt{a}$ and $a$, we can fix $\sigma_1$, $\sigma_2$ such that for all $x \in (\eta_1,\eta_2)$,
\[1 < \sigma_1 < \varchange a{x} < \sigma_2 < 2.\]
To handle the exponential sum estimates that follow, we will show that there exist positive constants $C_1$ and $C_2$ (depending on $a$ and $\gamma$) such that for all $h \in \N$, all sufficiently large $n \in \N$, and all $x \in [1,N_n-h]$,
\begin{align}\label{eqn:boundongdoubleprime} C_1 \ h \ \frac{(\log n)^{3}}{n^{3-\sigma_1}} \leq \left| g_{n,h}''(x) \right| \leq C_2 \ h \ \frac{(\log n)^{3}}{n^{3-\sigma_2}}.\end{align}
By the Mean Value Theorem, $g_{n,h}''(x) = h g_n'''(\xi_x)$ for some $\xi_x \in (x, x+h)$, so it suffices to show that for all $h \in \N$, all sufficiently large $n \in \N$, and all $x \in [1,N_n]$,
\begin{align}\label{eqn:boundongtripleprime} C_1 \ \frac{(\log n)^{3}}{n^{3-\sigma_1}} \leq \left| g_{n}'''(x) \right| \leq C_2 \  \frac{(\log n)^{3}}{n^{3-\sigma_2}}.\end{align}
Writing $g_n'''(x)$ explicitly reveals that
\[g_n'''(x) = g_n(x) \left( \frac{\log n}{n} \right)^3 \left( \frac{x+\bothfloor {\eta_1 n}}{n} \right)^{-3} \left(\frac{\varchange a {\frac{x+\bothfloor {\eta_1 n}}{n}}^6}{-(\log a)^3} + \frac{r(x)}{\log n} \right),\]
where $|r(x)| \ll_a 1$ for $x \in [1,N_n]$ because $(x+\bothfloor {\eta_1 n})/ n \in (\eta_1,\eta_2)$. The inequality in (\ref{eqn:boundongtripleprime}) follows for $n$ sufficiently large since $|\gamma|n^{\sigma_1} \leq |g_n(x)| \leq |\gamma|n^{\sigma_2}$.

To prove (\ref{eqn:equidistributiontoshow}), it suffices by Weyl's Criterion (\cite{kuipersbook}, Chapter 1, Theorem 2.1) to show that for all $b \in \Z \setminus \{0\}$,
\[\frac 1{N_n} \sum_{i=1}^{N_n} e \big(b g_n(i) \big) \longrightarrow 0 \quad \text{as} \quad n \to \infty,\]
where $e(x) = e^{2 \pi i x}$. By the van der Corput Difference Theorem (\cite{kuipersbook}, Chapter 1, Theorem 3.1) and another application of Weyl's Criterion, it suffices to prove that for all $h \in \N$ and for all $b \in \Z \setminus \{0\}$,
\begin{align}\label{eqn:vdcorputdifference}\frac 1{N_n-h} \sum_{i=1}^{N_n-h} e \big(b g_{n,h}(i) \big) \longrightarrow 0 \quad \text{as} \quad n \to \infty.\end{align}
An exponential sum estimate (\cite{kuipersbook}, Chapter 1, Theorem 2.7) gives us that
\[\frac 1{N_n-h} \left| \sum_{i=1}^{N_n-h} e \big(b g_{n,h}(i) \big) \right| \leq \left( \frac{|b| \ |g_{n,h}'(N_n-h) - g_{n,h}'(1) | +2}{N_n-h} \right) \left( \frac{4}{\sqrt{|b| \rho}} +3 \right),\]
where $\rho = C_1 h (\log n)^{3} \big/ n^{3-\sigma_1}$ from (\ref{eqn:boundongdoubleprime}). By the Mean Value Theorem and the upper bound from (\ref{eqn:boundongdoubleprime}), we see the right hand side is bounded from above for sufficiently large $n$ by
\[\left( |b|C_2 h \ \frac{(\log n)^{3}}{n^{3-\sigma_2}} + \frac 2{N_n-h} \right) \left( \frac{4n^{(3-\sigma_1)/2}}{\sqrt{|b| C_1 h (\log n)^{3}}} +3 \right) \ll \frac{n^{(3-\sigma_1)/2}}{n \sqrt{|b| h (\log n)^{3}}},\]
where the implicit constant depends on $a$, $\gamma$, $\eta_1$, and $\eta_2$. The limit in (\ref{eqn:vdcorputdifference}) follows since $(3-\sigma_1)/2 < 1$.
\end{proof}

\section{Remarks and conclusions}\label{sec:remarks}

Here are a number of natural questions and further directions.

\begin{enumerate}[i.]

\item Equations of the form $y=a x- b$ where $0 \leq b < 1$ are handled by Theorem \ref{thm:maintheorem} by rearranging, but there are still simple linear equations which the theorem does not handle.  For example, is the equation $y=2x+2$ solvable in $\psa$ for Lebesgue-a.e. $1< \al< 2$?

\begin{question}Does Theorem \ref{thm:maintheorem} hold with the assumption $0 \leq b < a$ in part \emph{i.} replaced by $a \notin \{0, 1\}$?\end{question} The answer is likely `yes.' Using the technique above, we need infinitely many $n$'s for which the interval $a^{1/ \al}n + [L_n,R_n)$ contains an integer, where now $L_n$ and $R_n$ are both positive or both negative. Accounting for $\{n^\al\}$ and changing variables, this requires control on the set of $n$'s for which $\fracpart {\theta n}$ falls within a shrinking annulus about $0$. These shrinking annuli are not nested, making them difficult to handle with the established theory.

\item Is there a non-metrical version of Theorem \ref{thm:maintheorem}?

\begin{question}Does Theorem \ref{thm:maintheorem} hold with ``Lebesgue-a.e.'' replaced by ``all''?\end{question} The inequality in (\ref{eqn:systemone}) is solvable in $\N$ when $a^{1/\al}$ is irrational by Dirichlet's Theorem, and the whole system is solvable in $\N$ when $a^{1/\al}$ is rational since $\fracpart {n^{\al}}$ is uniformly distributed along arithmetic progressions containing $0$. Therefore, exceptional $\al$'s would only arise because of the second condition, the ``twist,'' in (\ref{eqn:systemone}) when $a^{1/\al}$ is irrational.

Here are two thoughts for proving the result for all $1< \al< 2$. The result would be immediate from Theorem \ref{thm:firstDAform} if $( n^\al)_{n \in \N}$ was known to be equidistributed (or even dense) modulo 1 along denominators of the continued fraction convergents of $a^{1/\al}$. Alternatively, perhaps the set $\psa$ is sufficiently pseudorandom as to contain solutions to linear equations; see \cite{conlongowers} for recent results regarding combinatorial structure in sparse random sets.

\item Asymptotics are known for the distribution of Piatetski-Shapiro sequences in arithmetic progressions, the square-free numbers, and the primes; it is feasible that analogous asymptotics hold for the number of solutions to linear equations, as well.
\begin{question}Is it true that for Lebesgue-a.e. or for all $1 < \al < 2$,
\[\left|\big\{1 \leq m,n \leq N \ \big| \ \intpart {m^\al} = a \intpart {n^\al} + b \big\}\right| \sim_{a,b,\al} N^{2-\al}?\]
\end{question}
\noindent It is not hard to verify that $N^{2-\al}$ is the correct asymptotic for $\al \leq 1$.

\item Which systems of linear equations are solvable in $\psa$? Consider, for example, the system $y=2x$, $z=3x$. Just as is done in Section \ref{sec:reduction}, this can be reduced after a change of variables to the system
\begin{align*}\begin{dcases} \fracdist {\theta n} \leq \frac{c}{n^{\varchange a\theta -1}} \\ \fracdist {\theta^{\log_3 2} n} \leq \frac{c}{n^{\varchange a\theta -1}} \\ \fracpart {\ga n^{\varchange a\theta }} \in I \end{dcases}.\end{align*}
This is ``twisted'' Diophantine approximation on the curve $x \mapsto (x,x^{\log_3 2})$; see \cite{beresnevich}, Theorem 1. Assuming the twist does not interfere with the approximation, it is conceivable that this system is solvable for Lebesgue-a.e. $1 < \al < 3/2$.

\item Solving the system $y=2x$, $z=3x$ in $\psa$ is the same as finding $\fs {(x,x,x)}$ in $\psa$; recall (\ref{eqn:fsdef}). It is an open problem (\cite{guybook}, D18) to determine whether or not the set of squares contains a set of the form $\fs {(x_i)_{i=1}^3}$. We can use Theorem \ref{thm:maintheorem} to solve this problem in almost all $\psa$.
\begin{corollary}
For Lebesgue-a.e. $1 < \al < 2$, the set $\psa$ contains infinitely many sets of the form $\fs {(x_i)_{i=1}^3}$.
\end{corollary}

\begin{proof}
By Theorem 1, there are infinitely many $x \in \psa$ for which $2x \in \psa$. For $x$ sufficiently large (depending on $\al$), the set $\psa$ contains an arithmetic progression of step $x$ and length $3$ (\cite{frantzikinakiswierdl}, Proposition 5.1). If $z$ starts such a progression, then $\fs {(x,x,z)} \subseteq \psa$.
\end{proof}

\begin{question}(V. Bergelson) Is $\psa$ an $\ipzero$ set for Lebesgue-a.e. or all $1 < \al < 2$?\end{question} V. Bergelson remarked that while $\psa$ may not always possess ``linear structure,'' it may contain higher order structure. Indeed, the set $\ps {m/n}$ contains the set of $m^\text{th}$ powers, and this implies that $\ps{m/n}$ is a set of multiple recurrence; see \cite{BL}.

The set $A \subseteq \N$ possesses \emph{$\vip$-structure}\footnote{For a more general definition, see \cite{BFM, BM}.} if it contains arbitrarily large subsets of the form
\[\left\{ f\left( \sum_{i \in I} x_i^{(1)}, \ldots, \sum_{i \in I} x_i^{(k)} \right) \ \middle | \ I \subseteq \{1, \ldots, n\}, \ I \neq \emptyset \right\},\]
where $f \in \Z[z_1, \ldots, z_k]$ has zero constant term and $(x_i^{(1)})_{i=1}^n, \ldots, (x_i^{(k)})_{i=1}^n \subseteq \N$. Thus, the set of $m^\text{th}$ powers possesses $\vip$-structure. Note that when $\deg(f) = 1$, the set above is a finite sums set. Recall from Section \ref{sec:intro} that any set with $\vip$-structure is a set of multiple recurrence; perhaps $\vip$-structure in $\psa$ gives an alternate explanation of the set's recurrence properties. \begin{question} (V. Bergelson) Does $\psa$ possess $\vip$-structure for all $\al > 1$?\end{question}

\item Theorem \ref{thm:maintheorem} gives that for many linear equations, $\al=2$ is a threshold value for being solvable or unsolvable in $\psa$. Do other linear equations have such a threshold, and can we compute it?  \begin{question}Does there exist an $\al_S > 1$ with the property that for Lebesgue-a.e. or all $\al > 1$, the equation $x+y=z$ is solvable or unsolvable in $\psa$ according as $\al < \al_S$ or $\al > \al_S$?\end{question} As mentioned in Section \ref{sec:intro}, the equation $x+y=z$ is solvable in $\psa$ for all $\al < 2$; when $\al \geq 3$ is an integer, the equation is unsolvable in $\psa$. What happens for $\al$ just larger than $2$? The same question is meaningful and interesting for more general (systems of) linear equations.

\end{enumerate}

\bibliographystyle{abbrv}
\bibliography{psbib}

\end{document}